\documentclass[11pt]{article}

\usepackage[a4paper,margin=1.15in]{geometry}
\usepackage[T1]{fontenc}
\usepackage[utf8]{inputenc}
\usepackage{amsmath,amssymb,amsfonts,amsthm,mathtools}
\usepackage{enumitem}
\usepackage{mathrsfs}
\usepackage[hidelinks]{hyperref}

\theoremstyle{plain}
\newtheorem{theorem}{Theorem}[section]
\newtheorem{lemma}[theorem]{Lemma}

\theoremstyle{definition}
\newtheorem{definition}[theorem]{Definition}
\newtheorem{convention}[theorem]{Convention}

\newtheorem{question}[theorem]{Question}

\theoremstyle{remark}

\newcommand{\Pp}{\mathbb P}
\newcommand{\Qq}{\mathbb Q}
\newcommand{\Kk}{\mathbb K}

\newcommand{\Rr}{\mathbb R}
\newcommand{\Ll}{\mathbb L}

\newcommand{\dotQ}{\dot{\mathbb Q}}

\newcommand{\MA}{\mathsf{MA}}
\newcommand{\BPFA}{\mathsf{BPFA}}
\newcommand{\ZFC}{\mathsf{ZFC}}
\newcommand{\ZF}{\mathsf{ZF}}
\newcommand{\CH}{\mathsf{CH}}
\newcommand{\GCH}{\mathsf{GCH}}

\newcommand{\cof}{\operatorname{cof}}
\newcommand{\Lim}{\operatorname{Lim}}

\newcommand{\PsiThree}{\Psi_{3}}
\newcommand{\WO}{\operatorname{WO}}
\newcommand{\supp}{\operatorname{supp}}

\newcommand{\Wground}{W}

\title{Martin's Axiom, Large Continuum and Global $\Sigma^1_n$-Uniformization}
\author{Stefan Hoffelner\footnote{The author's research was funded in whole by the Austrian Science Fund (FWF), Grant DOI 10.55776/P37228. For the purpose of open access, the author has applied a CC BY public copyright license to any Author Accepted Manuscript version arising from this submission.}\\
Set Theory Research Group, Institute of Discrete Mathematics and Geometry, TU Wien\\
Wiedner Hauptstrasse 8--10/104, 1040 Wien, Austria\\
\texttt{stefan.hoffelner@tuwien.ac.at}}
\date{}

\begin{document}

\maketitle

\begin{abstract}
We construct a generic extension of $L$ satisfying Martin's Axiom,
$2^{\aleph_0}=\aleph_3$, a lightface $\Delta^1_3$ wellorder of the reals, and
$\Sigma^1_n$-uniformization for every $n\geq 2$ simultaneously. 
\end{abstract}


\section{Introduction}

Uniformization is a central problem in the descriptive set theory of definable
sets of reals.  Given \(A\subseteq\Rr\times\Rr\), a uniformization of \(A\) is a
function whose graph is contained in \(A\) and whose domain is the projection of
\(A\) on the first coordinate.  The projective uniformization problem asks
whether a projective set in the plane admits a uniformization of the same
projective complexity.  The Kondo--Novikov theorem gives the first positive
theorem in \(\ZFC\): coanalytic sets, and hence also \(\Sigma^1_2\) sets, admit
same-level uniformizations; see, for example, \cite{Kechris}.  Beyond this level
the behaviour of uniformization depends strongly on the ambient universe.

For a long time, the only general method for obtaining global
\(\Sigma\)-uniformization in models of \(\ZFC\) came from a good projective
wellorder of the reals.  Addison showed that a good \(\Delta^1_n\) wellorder
yields \(\Sigma^1_m\)-uniformization for all \(m\geq n\)
\cite{Addison}.  Thus \(L\) satisfies the global
\(\Sigma\)-uniformization pattern by means of its canonical good projective
wellorder.  This method is powerful but restrictive: the resulting universes are
close to the constructible universe and satisfy a strong definable version of
the continuum hypothesis.

Recent work has produced a different mechanism.  Instead of deriving
uniformization from an already available good wellorder, one can force
projective predicates which copy the truth values needed for uniformization.
This makes it possible to study global \(\Sigma\)-uniformization in much richer
universes of sets of reals, for example in the presence of forcing axioms, large
continuum, or prescribed behaviour of cardinal characteristics.  The present
paper belongs to this line of research.  It combines the
branch-versus-specialize coding apparatus of Fischer--Friedman--Zdomskyy with a
copying construction for projective truth values, and obtains global
\(\Sigma\)-uniformization together with Martin's Axiom and continuum
\(\aleph_3\).

This work is part of a broader program investigating separation, reduction and
uniformization in non-\(\mathsf{PD}\) universes.  Earlier contributions include
forcing the \(\Sigma^1_3\)-separation property
\cite{HoffelnerSeparation}, forcing \(\Pi^1_n\)-uniformization
\cite{HoffelnerPiUniformization}, separating \(\Pi^1_3\)-reduction from
\(\Pi^1_3\)-uniformization \cite{HoffelnerPiReduction}, and obtaining
\(\Sigma^1_3\)-uniformization from forcing axioms
\cite{HoffelnerForcingAxioms}.  The global copying method for
\(\Sigma\)-uniformization was developed further in the \(\BPFA\) setting in
\cite{BPFA_and_global_Sigma-uniformization}.  Related recent work studies
separation versus reduction, projective wellorders together with
\(\Pi\)-side uniformization, mixtures of upper \(\Sigma\)-uniformization with
lower \(\Pi\)-side phenomena, and distinctions between boldface and lightface
uniformization at low projective levels; see
\cite{HoffelnerSeparationReduction,HoffelnerCHPiUniformization,HoffelnerUpperSigma,HoffelnerSigma34}.

\begin{theorem}[Main theorem]
There is a generic extension of \(L\) satisfying
\[
\MA+2^{\aleph_0}=\aleph_3
\]
such that the reals have a lightface \(\Delta^1_3\) wellorder and, for every
\(n\geq 2\), every boldface \(\Sigma^1_n\) set of pairs of reals admits a
boldface \(\Sigma^1_n\) uniformization.
\end{theorem}

Thus global \(\Sigma\)-uniformization is compatible with a universe which is far
from \(L\) in several respects.  The model satisfies Martin's Axiom and has
large continuum.  Moreover, \(\MA+\neg\CH\) gives regularity at the second
projective level: every boldface \(\Sigma^1_2\) set of reals is Lebesgue
measurable and has the Baire property; see, for example, \cite{Fremlin,Kechris}.
At the same time, the model has a lightface \(\Delta^1_3\) wellorder of the
reals, and hence is far from the determinacy universe.  The theorem therefore
places global same-level \(\Sigma\)-uniformization in a new forcing-axiom
environment, rather than in the classical constructible setting.

The forcing construction starts from the Fischer--Friedman--Zdomskyy machinery.
Their model of Martin's Axiom with \(2^{\aleph_0}=\aleph_3\) and a lightface
\(\Delta^1_3\) wellorder uses a matrix of Suslin trees, where information is
recorded by deciding, for each tree in a block, whether to specialize the tree or
to add a cofinal branch through it \cite{FISCHER2013763}.  We use the same
apparatus in two separate ways.  One family of blocks is reserved for the
projective wellorder.  A disjoint family of blocks is reserved for the
uniformization construction.

The uniformization part is based on a lightface \(\Sigma^1_3\) predicate
\(\PsiThree(z)\), whose positive instances are exactly the reals deliberately
inserted into the uniformization-reserved branch-versus-specialize blocks.  From
this predicate we build higher projective predicates by alternating real
quantifiers over the assertions that a tuple is coded, or is not coded, by
\(\PsiThree\).  At a uniformization stage the construction considers triples
\[
(x,y_\alpha,a^\alpha_0)
\]
ordered by the final projective wellorder and makes the least successful
candidate projectively recognizable.

The main point is that the construction does not directly define the
minimality condition.  A naive definition saying that no earlier \(y'\) works
would have the wrong projective complexity.  Instead, the iteration copies the
relevant matrix truth values into the coding predicate.  Once the bookkeeping has
produced a tuple of reals, the remaining matrix is only \(\Pi^1_2\) or
\(\Sigma^1_2\), and its truth value is absolute between the intermediate models
and the final extension.  External maps
\[
\pi_\alpha:(2^\omega)^\alpha\longrightarrow 2^\omega
\]
are used only as bookkeeping devices to compress the family of earlier failures
or later exclusions into one real.  The final projective definition does not
mention these maps; it only asks whether the corresponding compressed real has
been coded by \(\PsiThree\).

This produces a projective copy of the relevant infinitary minimality
information.  From outside the construction, the bookkeeping arranges a possibly
infinite conjunction, or dually a possibly infinite disjunction, of
\(\Pi^1_2\)- or \(\Sigma^1_2\)-matrix tests.  From inside the final model, this
information is read by the ordinary projective predicate \(\PsiThree\), and hence
has the correct complexity.

Unless explicitly stated otherwise, projective pointclasses are understood in the
boldface sense.  Lightface definability is always indicated explicitly.  The
Fischer--Friedman--Zdomskyy apparatus and its projective wellorder predicate are
recalled in Section~\ref{sec:ffdz}.  Section~\ref{sec:psi-three} isolates the
base predicate \(\PsiThree\).  Section~\ref{sec:compression} records the
external compression facts used by the bookkeeping.  The final iteration is
defined in Section~\ref{sec:iteration}; it partitions \(\omega_3\) into stages
for Martin's Axiom, for the wellorder, and for global
\(\Sigma\)-uniformization.  The copied projective predicates and the
uniformization proof are given in Sections~\ref{sec:uniformization} and
\ref{sec:uniformization-final}.  A cardinal-characteristic variant is recorded in
Section~\ref{sec:cardinal-variant}.

\section{The Fischer--Friedman--Zdomskyy apparatus}\label{sec:ffdz}

We recall the part of \cite{FISCHER2013763} needed for this article.  All forcing is
performed over $L$ until the preliminary forcing has been introduced.  After the
preliminary forcing, if $G_0\subseteq\Pp_0$ is generic over $L$, we write
\[
\Wground=L[G_0]
\]
for the resulting extension.  If $\Pp_\alpha$ is a later c.c.c. iteration over
$\Wground$ and $G_\alpha\subseteq\Pp_\alpha$ is generic, intermediate extensions
are denoted by $\Wground[G_\alpha]$.  We shall not use the notation $V_\alpha$ for
these extensions.

We also fix here, once and for all, a definable partition of the limit ordinals
below $\omega_3$ into three disjoint cofinal classes
\[
\mathcal L_{\mathrm{MA}},\qquad \mathcal L_{\mathrm{WO}},\qquad
\mathcal L_{\Sigma}.
\]
The definition is made in $L$ and is absolute to suitable initial segments in the
obvious sense: a suitable model uses its internal version of the same canonical
partition.  The class $\mathcal L_{\mathrm{WO}}$ is reserved for the FFDZ wellorder
blocks, while $\mathcal L_{\Sigma}$ is reserved for the uniformization coding
blocks.  The remaining c.c.c. bookkeeping for Martin's Axiom is placed on
$\mathcal L_{\mathrm{MA}}$ and on successor stages assigned to it.  The exact
choice of the partition is irrelevant; only definability, pairwise disjointness
and cofinality are used.

Fix a canonical sequence
\[
\vec S=\langle S_\xi:1<\xi<\omega_3\rangle
\]
of stationary subsets of $\omega_2\cap\cof(\omega_1)$, $\Sigma_1$-definable over
$L_{\omega_3}$ with parameter $\omega_2$.  Fix also a nicely definable
almost-disjoint family
\[
\vec B=\langle B_\xi:\xi<\omega_2\rangle
\]
of subsets of $\omega_1$, $\Sigma_1$-definable over $L_{\omega_2}$ with parameter
$\omega_1$.  For $\xi<\omega_3$, let $W_\xi$ denote the $L$-least subset of
$\omega_2$ coding $\xi$.  We also fix the definable almost-disjoint family
$\vec C=\langle C_{(\xi,\eta)}:\xi<\omega_1,\eta<\omega\cdot 3\rangle$ of subsets
of $\omega$ used at the final real-coding stages of the FFDZ construction.

\begin{definition}
A transitive model $M$ of $\ZF^-$ is \emph{suitable} if $\omega_2^M$ exists and
\[
\omega_2^M=\omega_2^{L^M}.
\]
Consequently $\omega_1^M=\omega_1^{L^M}$.
\end{definition}

Throughout, $\omega_1$-trees are identified with subsets of $\omega_1$ by the
$L$-least bijection between $\omega^{<\omega_1}$ and $\omega_1$.

\subsection{The preliminary mixed-support forcing}

For $0<\alpha<\omega_3$ and $n\in\omega$, let $\Kk^0_{\omega\cdot\alpha+n}$ be the
standard countably closed Jech forcing adding a Suslin tree
$T_{\omega\cdot\alpha+n}$ with countable conditions; see, for example,
\cite{Jech}.  Put
\[
\Kk_{0,\alpha}=\prod_{n\in\omega}\Kk^0_{\omega\cdot\alpha+n}
\]
with full support.

In the extension by $\Kk_{0,\alpha}$, the tree $T_{\omega\cdot\alpha+n}$ is coded
by killing stationarity of $S_{\omega_1\cdot(\omega\cdot\alpha+n)+\gamma}$ exactly
for $\gamma\in T_{\omega\cdot\alpha+n}$.  More precisely, let
\[
\Kk^1_{\alpha,n,\gamma}=\begin{cases}
\text{the forcing adding a club }C_{\omega_1\cdot(\omega\cdot\alpha+n)+\gamma}
\subseteq\omega_2\setminus S_{\omega_1\cdot(\omega\cdot\alpha+n)+\gamma},&
\gamma\in T_{\omega\cdot\alpha+n},\\
\text{the trivial forcing},&\gamma\notin T_{\omega\cdot\alpha+n}.
\end{cases}
\]
Let
\[
\Kk_{1,\alpha,n}=\prod_{\gamma<\omega_1}\Kk^1_{\alpha,n,\gamma},\qquad
\Kk_{1,\alpha}=\prod_{n\in\omega}\Kk_{1,\alpha,n},
\]
again with full support.

For a set $X$ of ordinals write
\[
0(X)=\{\eta:3\eta\in X\},\quad I(X)=\{\eta:3\eta+1\in X\},\quad
II(X)=\{\eta:3\eta+2\in X\}.
\]
Let $\chi:\omega_1\times\omega_2\to\omega_2$ be the fixed definable bijection.  In
$L^{\Kk_{0,\alpha}*\Kk_{1,\alpha}}$, let
$D_{\omega\cdot\alpha+n}\subseteq\omega_2$ code
\[
W_{\omega\cdot\alpha+n},\quad W_{\omega\cdot\alpha},\quad
\langle C_{\omega_1\cdot(\omega\cdot\alpha+n)+\gamma}:\gamma\in
T_{\omega\cdot\alpha+n}\rangle
\]
by requiring
\[
0(D_{\omega\cdot\alpha+n})=W_{\omega\cdot\alpha+n},\quad
I(D_{\omega\cdot\alpha+n})=W_{\omega\cdot\alpha},
\]
and
\[
II(D_{\omega\cdot\alpha+n})=
\chi\bigl(\{\langle\gamma,\eta\rangle:\gamma\in T_{\omega\cdot\alpha+n},\ 
\eta\in C_{\omega_1\cdot(\omega\cdot\alpha+n)+\gamma}\}\bigr).
\]
Choose $Z_{\omega\cdot\alpha+n}\subseteq\omega_2$ with even part
$D_{\omega\cdot\alpha+n}$ and with the usual localization property: if
$\beta<\omega_2$ is $\omega_2^M$ for a suitable $M$ and
$Z_{\omega\cdot\alpha+n}\cap\beta\in M$, then $\beta$ lies in the relevant club of
elementary submodels used in the FFDZ construction.

Let $\psi_{\mathrm{stat}}(\omega_1,\omega_2,Z,T,Z')$ be the $\Sigma_1$ formula
saying that $Z,Z'$ decode, through $\vec B$, canonical ordinal codes and clubs
disjoint from the appropriate stationary sets, as in \cite[Section 2]{FISCHER2013763}.
Let $\varphi_{\mathrm{loc}}(\omega_1,\omega_2,X,T,X')$ assert that, using
$\vec B$, $X$ and $X'$ almost-disjointly code $Z$ and $Z'$ such that
$\psi_{\mathrm{stat}}(\omega_1,\omega_2,Z,T,Z')$ holds.

The third component $\Kk_{2,\alpha,n}$ adds
$X_{\omega\cdot\alpha+n}\subseteq\omega_1$ almost-disjointly coding
$Z_{\omega\cdot\alpha+n}$ by conditions
$(s,s^*)\in[\omega_1]^{<\omega_1}\times[Z_{\omega\cdot\alpha+n}]^{<\omega_1}$,
ordered by end-extension of $s$ and the requirement that new points avoid
$B_\xi$ for $\xi\in s^*$.  Let
\[
\Kk_{2,\alpha}=\prod_{n\in\omega}\Kk_{2,\alpha,n}
\]
with full support.

Finally, define the localization forcing $\Ll(X,T,X')$ to consist of all
functions $r:|r|\times 3\to 2$ such that $|r|$ is a countable limit ordinal and:
\begin{enumerate}[label=(\arabic*)]
\item for $\gamma<|r|$, $\gamma\in X$ iff $r(\gamma,0)=1$;
\item for $\gamma<|r|$, $\gamma\in X'$ iff $r(\gamma,1)=1$;
\item for every $\gamma\leq |r|$ and every suitable model $M$ containing
$r\upharpoonright(\gamma\times 3)$, the model $M$ satisfies
\[
\varphi_{\mathrm{loc}}(\omega_1,\omega_2,X\cap\gamma,T\cap\gamma,X'\cap\gamma).
\]
\end{enumerate}
The order is end-extension.  Put
\[
\Kk_{3,\alpha,n}=\Ll(X_{\omega\cdot\alpha+n},T_{\omega\cdot\alpha+n},
X_{\omega\cdot\alpha})
\]
for $\alpha>0$, with the obvious trivial definition at $\alpha=0$, and set
\[
\Kk_{3,\alpha}=\prod_{n\in\omega}\Kk_{3,\alpha,n}.
\]
Let
\[
\Kk_\alpha=\Kk_{0,\alpha}*\Kk_{1,\alpha}*\Kk_{2,\alpha}*\Kk_{3,\alpha}.
\]

For $p\in\prod_{\alpha\in I}\Kk_\alpha$, define
\[
\supp_\omega(p)=\{\langle i,\alpha\rangle:i\in\{0,2,3\},\ p_{i,\alpha}\neq\mathbf 1\}
\]
and
\[
\supp_{\omega_1}(p)=\{\langle 1,\alpha,n,\zeta\rangle:p_{1,\alpha,n,\zeta}\neq\mathbf 1\}.
\]
A condition has \emph{mixed support} if $|\supp_\omega(p)|\leq\omega$ and
$|\supp_{\omega_1}(p)|\leq\omega_1$.  Let $\Pp_0$ be the mixed-support suborder of
$\prod_{\alpha<\omega_3}\Kk_\alpha$.

\begin{theorem}[Fischer--Friedman--Zdomskyy]\label{thm:ffdz-preliminary}
The forcing $\Pp_0$ is $\omega$-distributive, has the $\omega_3$-chain condition,
preserves cardinals, and preserves the relevant stationary sets in the following
precise sense: if $\gamma\notin T_{\omega\cdot\alpha+n}$, then
$S_{\omega_1\cdot(\omega\cdot\alpha+n)+\gamma}$ remains stationary in
$L^{\Pp_0}$.
\end{theorem}

\begin{proof}
This is the preliminary-stage analysis of \cite[Propositions 2.2, 2.3 and
2.6]{FISCHER2013763}.  The distributivity proof uses the localization component
$\Kk_3$ and the least-countable-elementary-submodel argument.  The preservation
of stationarity for $\gamma\notin T_{\omega\cdot\alpha+n}$ is obtained by the
mixed-support genericity construction with the relation $\leq^*$.
\end{proof}

\subsection{The projective wellorder predicate}\label{subsec:wo-predicate}

The FFDZ coding stage defines a finite-support c.c.c. iteration over $\Wground$.
At wellorder stages, a pair of reals is encoded by specializing some trees in a
fresh $\omega$-block and adding branches through the remaining trees.  We recall
the corresponding projective definition, since it is the prototype for the base
coding predicate used later.

Fix a recursive injection
\[
\Delta_{\mathrm{WO}}:2^\omega\longrightarrow\mathcal P(\omega)
\]
which is used in the FFDZ wellorder coding; for a pair $x,y$ write $x*y$ for the
usual recursive join and use $\Delta_{\mathrm{WO}}(x*y)$ as the block pattern.
Let $\mathcal L_{\mathrm{WO}}\subseteq\Lim(\omega_3)$ be the definable class of
limit stages reserved for wellorder coding.

Define $\WO(x,y)$ to be the assertion that there is a real $R$ such that for
every countable suitable model $M$ with $R,x,y\in M$, there is a limit ordinal
$\bar\alpha\in\mathcal L_{\mathrm{WO}}^M$ such that, for every $n\in\omega$, the
set
\[
T^M_{\bar\alpha,n}=\{\gamma<\omega_1^M:
S^M_{\omega_1\cdot(\omega\cdot\bar\alpha+n)+\gamma}
\text{ is nonstationary in }M\}
\]
is an $\omega_1^M$-tree, and $M$ sees that
\[
\begin{array}{rcl}
n\in\Delta_{\mathrm{WO}}(x*y)&\Longrightarrow& T^M_{\bar\alpha,n}\text{ is specialized},\\
n\notin\Delta_{\mathrm{WO}}(x*y)&\Longrightarrow& T^M_{\bar\alpha,n}\text{ has a cofinal branch}.
\end{array}
\]
Here all references to $\vec S$ and to the coding of trees are interpreted using
the canonical definitions inside $M$.

\begin{lemma}\label{lem:wo-sigma13}
The relation $\WO(x,y)$ is $\Sigma^1_3$.  In the FFDZ extension, $\WO(x,y)$ holds
if and only if $x$ precedes $y$ in the canonical stage-of-appearance wellorder of
the reals.  Consequently the wellorder is lightface $\Delta^1_3$.
\end{lemma}

\begin{proof}
The displayed definition has the form $\exists R\,\forall M\,\vartheta(R,x,y,M)$,
where $M$ ranges over reals coding countable suitable transitive models and
$\vartheta$ is arithmetical in the code for $M$ together with the canonical
$L$-definitions.  Hence it is $\Sigma^1_3$.

The correctness is exactly the argument of \cite[Lemma 2.8]{FISCHER2013763},
with the harmless change that the formula is explicitly restricted to the
wellorder-reserved class $\mathcal L_{\mathrm{WO}}$.  If the pair $(x,y)$ is
coded at a wellorder stage, the final almost-disjoint real $R$ carries the
specializing functions, branches and localized tree information into every
countable suitable model, so $\WO(x,y)$ holds.  Conversely, if $\WO(x,y)$ holds,
the stationarity preservation theorem from the preliminary forcing and the
factorization analysis of the later c.c.c. iteration imply that the decoded block
must be a genuine wellorder block.  The branch-versus-specialize pattern then
recovers exactly $\Delta_{\mathrm{WO}}(x*y)$, and hence the pair coded at that
stage is $(x,y)$.

Since $x<_{G}y$ is $\WO(x,y)$ and $y<_{G}x$ is $\WO(y,x)$, the complement of the
strict wellorder is the union of equality and the reversed relation.  Thus the
wellorder is also $\Pi^1_3$, and hence $\Delta^1_3$.
\end{proof}

The following lemma spells out the preservation point used in the proof of the
FFDZ coding stage.  In \cite[Claim~2.7]{FISCHER2013763}, at a limit coding stage
\(\alpha=\omega\cdot\beta\), the forcing is factored as
\[
  \Pp_\alpha=(\Pp_{0,<\alpha}*\bar{\Pp}_\alpha)\times\Pp_{0,\geq\alpha},
\]
and it is concluded that the fresh trees \(T_{\omega\cdot\alpha+n}\), which
come from the tail preliminary factor, remain Suslin after the c.c.c. coding
iteration.  This conclusion is correct, but it should not be read as saying that
Jech's forcing remains countably closed after passing to the intermediate c.c.c.
extension.  Rather, one commutes the independent factors and uses the countable
closure of the Jech factor in the original ground model, together with the c.c.c.
of the later factor.  The lemma below is the precise product form of this
argument.  It is the only additional point about the FFDZ preservation analysis
that will be used later; the rest of the mixed-support analysis is exactly that
of \cite[Section~2, especially Propositions~2.2, 2.3, 2.6 and Claim~2.7]
{FISCHER2013763}.

\begin{lemma}\label{lem:jech-subtlety}
Let \(\mathbb J\) be Jech's forcing for adding a Suslin tree, and let
\(\dot T\) be the canonical \(\mathbb J\)-name for the generic tree.  If
\(\mathbb P\) is a c.c.c. forcing notion belonging to the ground model, then
\[
  \mathbb P\times\mathbb J\Vdash
  ``\dot T\text{ is a Suslin tree}''.
\]
Equivalently, after forcing with \(\mathbb P\), the old Jech forcing
\(\mathbb J\), although it need not remain countably closed, still forces its
canonical generic tree to be Suslin.  In the FFDZ notation this is applied with
\(\mathbb J\) one of the factors \(\Kk^0_{\omega\cdot\alpha+n}\) in the fresh
tail block and with \(\mathbb P\) the c.c.c. coding part which is supported below
\(\alpha\).
\end{lemma}

\begin{proof}
We write conditions in \(\mathbb P\times\mathbb J\) as pairs \((p,t)\), with
stronger conditions smaller.  This is exactly the commutation argument behind
\cite[Claim~2.7]{FISCHER2013763}.  The proof uses the countable closure of
\(\mathbb J\) only in the ground model.  We never claim that \(\mathbb J\)
remains countably closed after forcing with \(\mathbb P\).

We first record the product-fusion fact which will be used repeatedly.  Let
\(D\subseteq\mathbb P\times\mathbb J\) be dense open and let
\((p,t)\in\mathbb P\times\mathbb J\).  Then there are a condition
\(t'\leq_{\mathbb J}t\) and a maximal antichain \(A\subseteq\mathbb P\) below
\(p\) such that
\[
  (q,t')\in D \qquad\text{for every }q\in A.
\]
Indeed, recursively build an antichain \(\langle q_\xi:\xi<\eta\rangle\) below
\(p\) and a decreasing sequence \(\langle t_\xi:\xi\leq\eta\rangle\) in
\(\mathbb J\).  If the antichain constructed so far is not maximal below
\(p\), choose \(q_\xi\leq p\) incompatible with all earlier \(q_\zeta\)'s.  By
density of \(D\), strengthen \((q_\xi,t_\xi)\) to some \((q'_\xi,t_{\xi+1})\in D\),
and replace \(q_\xi\) by \(q'_\xi\).  At limit stages take a lower bound in
\(\mathbb J\).  Since \(\mathbb P\) is c.c.c., the recursion stops at a countable
stage.  A final lower bound in \(\mathbb J\) gives \(t'\), and openness of \(D\)
gives the displayed conclusion.

Now suppose, toward a contradiction, that some \((p_0,t_0)\) forces that
\(\dot A\) is an uncountable antichain in \(\dot T\).  Strengthening
\((p_0,t_0)\), we may assume that \(\dot A\) is forced to be a maximal antichain
of \(\dot T\).  Fix a sufficiently large regular \(\Theta\) and a countable
\(M\prec H_\Theta\) containing
\[
  \mathbb P,\mathbb J,\dot T,\dot A,p_0,t_0.
\]
Let \(\delta=M\cap\omega_1\).  We construct in the ground model a decreasing
sequence
\[
  t_0\geq_{\mathbb J} t_1\geq_{\mathbb J}t_2\geq_{\mathbb J}\cdots
\]
and, together with it, a countable set \(\mathcal B\) of branches through the
increasing union of the trees \(t_n\).  The bookkeeping of the construction
lists all nodes which occur in the successive conditions and all dense sets from
\(M\) relevant to deciding intersections with \(\dot A\).  When a node \(x\) is
considered, we use the preceding product-fusion fact, applied to the dense set
of conditions which decide a member of \(\dot A\) compatible with an extension of
\(x\).  This produces a lower Jech condition and a maximal antichain in the
\(\mathbb P\)-coordinate below \(p_0\).  Since each such antichain is countable
and only countably many nodes are treated, the construction remains countable;
at limit steps we use the ground-model countable closure of \(\mathbb J\).

Let \(t^*\) be a lower bound for the constructed sequence.  We may arrange the
construction so that every node of \(t^*\) lies on a branch \(b\in\mathcal B\),
and for each \(b\in\mathcal B\) there are a maximal antichain
\(A_b\subseteq\mathbb P\) below \(p_0\) and nodes \(y_{b,q}\in b\), for
\(q\in A_b\), such that
\[
  (q,t^*)\Vdash y_{b,q}\in\dot A.
\]
End-extend \(t^*\) in the usual Jech manner by adding one new node at level
\(\delta\) above each branch \(b\in\mathcal B\).  Call the resulting condition
\(t^+\).  We claim that
\[
  (p_0,t^+)\Vdash \dot A\subseteq t^* .
\]
Suppose not.  Then some \((r,s)\leq(p_0,t^+)\) forces that a node
\(z\in\dot A\setminus t^*\).  Let \(u\) be the new level-\(\delta\) predecessor
of \(z\) in \(t^+\), and let \(b\in\mathcal B\) be the branch below \(u\).  Pick
\(q\in A_b\) compatible with \(r\), and let \(r'\leq r,q\).  Since
\((q,t^*)\Vdash y_{b,q}\in\dot A\), the condition \((r',s)\) forces both
\(y_{b,q}\in\dot A\) and \(z\in\dot A\).  But \(y_{b,q}<_{\dot T}u\leq_{\dot T}z\),
contradicting that \(\dot A\) is an antichain.  Hence \((p_0,t^+)\) forces
\(\dot A\subseteq t^*\), and therefore forces \(\dot A\) to be countable.

The generic tree has no cofinal branch as well.  In Jech's forcing the generic
tree is normal and splitting by dense requirements depending only on the
\(\mathbb J\)-coordinate, so the same remains forced by the product.  A cofinal
branch through a normal splitting \(\omega_1\)-tree yields an uncountable
antichain by choosing, along the branch, incompatible splitting successors at
strictly increasing levels.  Since the product adds no uncountable antichain to
\(\dot T\), it adds no cofinal branch through \(\dot T\).

Finally, the product plainly forces that \(\dot T\) has height \(\omega_1\) and
countable levels, because each level is a countable level of some Jech condition
once it appears.  Thus \(\mathbb P\times\mathbb J\) forces \(\dot T\) to be
Suslin.
\end{proof}

It is useful to record the exact form in which the preceding lemma will be used.
This formulation is deliberately stated in terms of a local side forcing, rather
than in terms of the full iteration which will only be defined later.  Fix a
limit block \(\alpha\).  We say that a c.c.c. forcing \(\mathbb R\) is
\emph{supported below \(\alpha\)} if it belongs to the extension generated by the
preliminary coordinates below \(\alpha\), together with the already constructed
c.c.c. quotient below \(\alpha\), and if its definition uses no branch or
specialization forcing for any of the trees
\[
   T_{\omega\cdot\alpha+n}\qquad(n<\omega)
\]
and no almost-disjoint coding real attached to this block.  Equivalently, in the
FFDZ factorization at a fresh limit block, \(\mathbb R\) is a complete subforcing
of the c.c.c. side quotient which lives over
\(\Pp_{0,<\alpha}\), while the whole fresh preliminary tail
\(\Pp_{0,\geq\alpha}\) is still untouched.  This is the form in which all later
MA stages, wellorder stages below \(\alpha\), and uniformization stages below
\(\alpha\) will interact with the block \(\alpha\).  When a stage is assigned to
\(\alpha\) itself, it is no longer a side forcing in this sense; it is precisely
the deliberate coding action on the block.

\begin{lemma}\label{lem:no-accidental-codes}
Let \(\alpha\in\mathcal L_{\mathrm{WO}}\cup\mathcal L_\Sigma\) be a fresh block,
and let \(\mathbb R\) be a c.c.c. side forcing supported below \(\alpha\) in the
above sense.  Then, after forcing with \(\mathbb R\), the fresh block
\[
   \langle T_{\omega\cdot\alpha+n}:n<\omega\rangle
\]
which is added by the tail preliminary factor still consists of Suslin trees.
Consequently no branch-versus-specialize pattern appears on the block
\(\alpha\) before the construction deliberately acts on that block.
\end{lemma}

\begin{proof}
Work in the FFDZ factorization at the fresh block \(\alpha\).  The side forcing
\(\mathbb R\) is supported below \(\alpha\), so the fresh tree-adding coordinates
for \(T_{\omega\cdot\alpha+n}\) remain independent tail coordinates.  For each
fixed \(n<\omega\), the relevant part of the forcing factors, over the ground
model in which the preliminary apparatus is defined, as a product
\[
   \mathbb R\times\mathbb J_n,
\]
where \(\mathbb J_n\) is the Jech forcing adding
\(T_{\omega\cdot\alpha+n}\).  Lemma~\ref{lem:jech-subtlety} therefore implies
that \(\mathbb R\times\mathbb J_n\) forces the \(\mathbb J_n\)-generic tree to be
Suslin.  Thus, in the presentation where the side forcing is performed first,
the old Jech forcing still adds a Suslin tree, even though it need not be
countably closed over the side-forcing extension.

The remaining preliminary coordinates are handled exactly as in the FFDZ
mixed-support analysis.  In particular, the stationary-preservation argument for
unused coordinates prevents the localization apparatus from falsely interpreting
an untouched tree as having already been specialized or branched.  Hence a fresh
block remains genuinely fresh until the recursive construction assigns that
block to a wellorder or uniformization coding stage.  At that later stage the
appearance of the branch-versus-specialize pattern is intentional and is no
longer covered by the side-forcing hypothesis of the lemma.
\end{proof}

\section{A reusable branch-versus-specialize code}\label{sec:psi-three}

The wellorder predicate from Lemma~\ref{lem:wo-sigma13} will be kept separate
from the predicate used for uniformization.  We use the partition
\[
\mathcal L_{\mathrm{MA}},\qquad \mathcal L_{\mathrm{WO}},\qquad
\mathcal L_{\Sigma}
\]
fixed at the beginning of Section~\ref{sec:ffdz}: wellorder blocks live on
$\mathcal L_{\mathrm{WO}}$, while the reusable coding predicate below only looks
at blocks from $\mathcal L_{\Sigma}$.

Fix a recursive injection
\[
\Delta:2^\omega\longrightarrow\mathcal P(\omega)
\]
whose range omits $\omega$.  The omitted value $\omega$ will serve as a default
pattern for unused or dummy uniformization blocks.

For $\alpha\in\mathcal L_\Sigma$ and $z\in 2^\omega$, the block action
$\Qq^{\mathrm{blk}}(\alpha,z)$ is the finite-support iteration over $n\in\omega$
which specializes $T_{\omega\cdot\alpha+n}$ when $n\in\Delta(z)$ and forces with
$T_{\omega\cdot\alpha+n}$, thereby adding a cofinal branch, when
$n\notin\Delta(z)$.  The block is then almost-disjointly coded by a real
$R_\alpha$, exactly as in the FFDZ real-coding stage.  If a stage is required to
remain dummy, we specialize all trees in the block; this has pattern $\omega$ and
therefore does not code a real in the range of $\Delta$.

\begin{definition}[The base predicate]\label{def:psi3}
For a real $z$, let $\PsiThree(z)$ be the assertion that there is a real $R$ such
that for every countable suitable model $M$ with $R,z\in M$, there is
$\bar\alpha\in\mathcal L_\Sigma^M$ such that, for every $n\in\omega$, the tree
\[
T^M_{\bar\alpha,n}=\{\gamma<\omega_1^M:
S^M_{\omega_1\cdot(\omega\cdot\bar\alpha+n)+\gamma}
\text{ is nonstationary in }M\}
\]
is an $\omega_1^M$-tree and $M$ sees that
\[
\begin{array}{rcl}
n\in\Delta(z)&\Longrightarrow& T^M_{\bar\alpha,n}\text{ is specialized},\\
n\notin\Delta(z)&\Longrightarrow& T^M_{\bar\alpha,n}\text{ has a cofinal branch}.
\end{array}
\]
\end{definition}

\begin{lemma}\label{lem:psi3-complexity}
The predicate $\PsiThree(z)$ is lightface $\Sigma^1_3$.
\end{lemma}

\begin{proof}
The assertion has the form $\exists R\,\forall M\,\vartheta(R,z,M)$, where $M$
ranges over reals coding countable suitable transitive models and $\vartheta$ is
arithmetical in the code for $M$ together with the canonical $L$-definitions of
$\vec S$, the partition classes and the block decoding.  Thus the complexity is
$\Sigma^1_3$.
\end{proof}

Before the full iteration is defined, we record the local correctness statement for the
base predicate.  It applies to any forcing construction satisfying the block
discipline used below: every block from $\mathcal L_\Sigma$ is either acted on once with
pattern $\Delta(z)$ for a unique real $z$, or is made a dummy block by specializing every
tree in the block; once a block has been used or declared dummy, no later stage acts on it.

\begin{lemma}[Unary exactness of the base code]\label{lem:block-correctness}
Let $\Pp$ be a completed iteration satisfying this block discipline, and let $G$ be
$\Pp$-generic.  Then, in $V[G]$, for every real $z$,
\[
\PsiThree(z)
\]
holds if and only if $z$ was deliberately placed by the iteration into one of the
uniformization-reserved branch-versus-specialize coding blocks.
\end{lemma}

\begin{proof}
We first fix the notation which will be used in both directions.  If
$\alpha\in\mathcal L_{\Sigma}$ and $n\in\omega$, put
\[
E^G_{\alpha,n}=\{\gamma<\omega_1:
S_{\omega_1\cdot(\omega\cdot\alpha+n)+\gamma}\text{ is nonstationary in }V[G]\}.
\]
By the preliminary FFDZ coding and the preservation theorem for the stationary sets,
\[
E^G_{\alpha,n}=T_{\omega\cdot\alpha+n}.                 \tag{1}
\]
More precisely, if $\gamma\in T_{\omega\cdot\alpha+n}$, the preliminary stationary-killing
forcing added a club through the complement of
$S_{\omega_1\cdot(\omega\cdot\alpha+n)+\gamma}$; while if
$\gamma\notin T_{\omega\cdot\alpha+n}$, the corresponding stationary set remains
stationary through the preliminary forcing and through all later c.c.c. stages which do
not explicitly kill it.  Thus the nonstationarity pattern recovers exactly the tree
$T_{\omega\cdot\alpha+n}$.

For a uniformization block $\alpha$, define its final pattern $b_\alpha\subseteq\omega$ by
\[
n\in b_\alpha
\quad\Longleftrightarrow\quad
\text{the action at coordinate }n\text{ of block }\alpha\text{ is specialization},
\]
equivalently,
\begin{align*}
   n\notin b_\alpha
\quad\Longleftrightarrow& \quad
\text{the action at coordinate }n\text{ of block }\alpha\text{ is forcing with }
T_{\omega\cdot\alpha+n} \\& \quad \text{and hence adding a cofinal branch}. 
\end{align*}
The construction gives the following dichotomy for every $\alpha\in\mathcal L_{\Sigma}$:
\[
b_\alpha=\Delta(u)\text{ for the unique real }u\text{ deliberately coded at }
\alpha,                                      \tag{2a}
\]
or else the block is a default block and
\[
b_\alpha=\omega.                             \tag{2b}
\]
Since $\omega\notin\operatorname{ran}(\Delta)$ and $\Delta$ is injective, the value of
$b_\alpha$, when it belongs to $\operatorname{ran}(\Delta)$, determines the coded real
uniquely.

Suppose first that $z$ is deliberately coded at the uniformization block
$\alpha\in\mathcal L_{\Sigma}$.  Let $R=R_\alpha$ be the real added by the
almost-disjoint coding step of this block.  Thus $R$ codes the three sequences
\[
\langle A_{\omega\cdot\alpha+n}:n\in\omega\rangle,
\qquad
\langle Y_{\omega\cdot\alpha+n}:n\in\omega\rangle,
\qquad
\langle T_{\omega\cdot\alpha+n}:n\in\omega\rangle,
\]
where $A_{\omega\cdot\alpha+n}$ is a specializing function if $n\in\Delta(z)$ and a
cofinal branch through $T_{\omega\cdot\alpha+n}$ if $n\notin\Delta(z)$.

Let $M$ be a countable suitable model with $R,z\in M$.  Decoding $R$ inside $M$ gives
initial segments
\[
\langle A^M_n:n\in\omega\rangle,
\qquad
\langle Y^M_n:n\in\omega\rangle,
\qquad
\langle U^M_n:n\in\omega\rangle,
\]
and a block index $\beta^M\in M$ which $M$ regards as an element of
$\mathcal L_{\Sigma}$.  By the localization part of the preliminary forcing
(the $K_3$-step in the FFDZ apparatus), for every $n\in\omega$ the model $M$ satisfies
that the tree decoded from the stationary-kill pattern at the block
$\omega\cdot\beta^M+n$ is precisely $U^M_n$.  In the notation of Definition~\ref{def:psi3},
\[
M\models
T^M_{\omega\cdot\beta^M,n}=U^M_n.              \tag{3}
\]
Moreover the object $A^M_n$ decoded from $R$ has, in $M$, exactly the kind prescribed by
$\Delta(z)$:
\[
\begin{array}{rcl}
n\in\Delta(z)     &\Longrightarrow& M\text{ sees that }A^M_n\text{ specializes }T^M_{\omega\cdot\beta^M,n},\\[2mm]
n\notin\Delta(z) &\Longrightarrow& M\text{ sees that }A^M_n\text{ is a cofinal branch through }T^M_{\omega\cdot\beta^M,n}.
\end{array}                                      \tag{4}
\]
Equations (3) and (4) verify the matrix in Definition~\ref{def:psi3}.  Since $M$ was
arbitrary, the real $R_\alpha$ witnesses $\PsiThree(z)$.

Conversely suppose that $\PsiThree(w)$ holds in $V[G]$, and let $R$ witness it.
Choose a sufficiently large regular $\Theta$ and a countable elementary submodel
\[
N\prec H_\Theta^{V[G]}
\]
with
\[
R,w,\Pp_{\omega_3},G,\mathcal L_{\Sigma},\Delta\in N.
\]
Let $\pi:N\to M$ be the transitive collapse.  Since $R$ and $w$ are reals, the
collapse fixes them.  The model $M$ is a countable suitable transitive model
containing $R$ and $w$, so Definition~\ref{def:psi3} supplies, inside $M$, an
index $\bar\beta\in\mathcal L_{\Sigma}^M$ witnessing the required
branch-versus-specialize pattern for $w$.  Let $\beta=\pi^{-1}(\bar\beta)$ and
let $\alpha$ be the corresponding external preliminary block index.

Inside $M$, for every $n\in\omega$, the tree decoded from the nonstationarity
pattern at the block $\omega\cdot\bar\beta+n$ is specialized exactly when
$n\in\Delta(w)$ and has a cofinal branch exactly when $n\notin\Delta(w)$.  By the
preliminary FFDZ localization and condensation analysis, this decoded tree is the
collapse of an initial segment of the genuine preliminary tree
$T_{\omega\cdot\alpha+n}$.  Consequently the pattern seen in $M$ is the collapse
of the actual final branch-versus-specialize pattern $b_\alpha$.  Hence
\[
\forall n\in\omega\quad
\bigl(n\in b_\alpha\Longleftrightarrow n\in\Delta(w)\bigr),
\]
and therefore
\[
b_\alpha=\Delta(w).                            \tag{5}
\]
The index $\alpha$ belongs to the uniformization-reserved class, because the
witnessing formula explicitly requires the block index to lie in
$\mathcal L_\Sigma$.  It cannot be a wellorder block, since the classes
$\mathcal L_{\mathrm{WO}}$ and $\mathcal L_\Sigma$ are disjoint.  It cannot be an
unused fresh block: by Lemma~\ref{lem:no-accidental-codes}, before a fresh block
is acted upon all trees in that block remain Suslin, so no
branch-versus-specialize pattern can be decoded there.  It cannot be a default
block either, because default blocks have pattern $\omega$, while
$\omega\notin\operatorname{ran}(\Delta)$.

Thus the action at block $\alpha$ was a genuine coding action.  By (2a), there is a
unique real $u$ deliberately coded at $\alpha$, and
\[
b_\alpha=\Delta(u).
\]
Together with (5) and the injectivity of $\Delta$, this gives $u=w$.  Hence $w$ was
deliberately placed into the uniformization-reserved block $\alpha$, as required.
\end{proof}

\section{Bookkeeping compression and cardinal arithmetic}\label{sec:compression}

The uniformization construction needs to let the bookkeeping talk about long
sequences of reals, for example families of counter-witnesses indexed by an
initial segment of the final wellorder.  The role of compression is purely
external.  We use cardinal arithmetic to choose, while defining the iteration,
representatives for such sequences.  These representatives are not part of the
final projective definitions.

\begin{lemma}\label{lem:bounded-sequence-names}
Let $\langle\mathbb P_\xi,\dot{\mathbb Q}_\xi:\xi\leq\omega_3\rangle$ be a
finite-support c.c.c. iteration whose iterands have size $<\aleph_3$.  If
$\alpha<\omega_3$, then every $\mathbb P_{\omega_3}$-name for an
$\alpha$-sequence of reals is equivalent to a $\mathbb P_\beta$-name for some
$\beta<\omega_3$.
\end{lemma}

\begin{proof}
Code an $\alpha$-sequence of reals by a subset of $\alpha\times\omega$.  For each
coordinate $(\xi,n)\in\alpha\times\omega$, choose a countable maximal antichain
deciding the corresponding bit.  The union of the finite supports of all
conditions appearing in these antichains has size at most
$|\alpha|\cdot\aleph_0\leq\aleph_2$.  Since $\omega_3$ is regular, this union is
bounded in $\omega_3$.  Hence the name is equivalent to a name over a proper
initial segment of the iteration.
\end{proof}

\begin{lemma}\label{lem:final-cardarith}
For the forcing constructed below, the final model satisfies
\[
2^\omega=2^{\omega_1}=2^{\omega_2}=\aleph_3.
\]
Consequently, for every $0<\alpha<\omega_3$,
\[
(2^\omega)^\alpha=2^\omega.
\]
\end{lemma}

\begin{proof}
The preliminary mixed-support forcing has size $\aleph_3$, preserves cardinals,
and adds no reals.  The later finite-support c.c.c. iteration has length
$\omega_3$, all iterands have size $<\aleph_3$, and hence the full forcing has
size $\aleph_3$.  It adds cofinally many reals, so $2^\omega=\aleph_3$ in the
final model.

Working in $L$, the full forcing has the $\omega_3$-chain condition.  Hence every
name for a subset of $\omega_i$, $i\in\{1,2\}$, is equivalent to a nice name which,
for each $\xi<\omega_i$, uses an antichain of size $<\omega_3$.  Thus the number
of such names is at most
\[
(\aleph_3^{<\omega_3})^{\omega_i}=\aleph_3,
\]
using $\GCH$.  Since cardinals are preserved and $2^\omega=\aleph_3$, we get
$2^{\omega_1}=2^{\omega_2}=\aleph_3$.  Finally, if $0<\alpha<\omega_3$, then
$|\alpha|\leq\omega_2$, and therefore
\[
(2^\omega)^\alpha=\aleph_3^\alpha\leq \aleph_3^{\omega_2}=2^{\omega_2}=\aleph_3=2^\omega.
\]
The reverse inequality is trivial.
\end{proof}

\begin{convention}[External compression maps]\label{conv:compression-maps}
For every $0<\alpha<\omega_3$, fix externally, at the relevant stage of the
construction, a surjection
\[
\pi_\alpha:2^\omega\longrightarrow (2^\omega)^\alpha
\]
or equivalently a coding of $\alpha$-sequences of reals by single reals.  If
$\pi_\alpha(b)=\langle b^\eta:\eta<\alpha\rangle$, we write
$b^\eta$ for the $\eta$-th decoded real.  These maps are bookkeeping devices only
and are never mentioned in the final projective definitions.
\end{convention}

\begin{proof}[Justification of the convention]
Lemma~\ref{lem:final-cardarith} gives $|(2^\omega)^\alpha|=2^\omega$, and
Lemma~\ref{lem:bounded-sequence-names} ensures that every relevant final
sequence-name appears at some bounded stage.  Therefore the bookkeeping can
present a single real whose external decoding is the required sequence.
\end{proof}

\section{The final iteration}\label{sec:iteration}

We now define the finite-support c.c.c. iteration over $\Wground=L[G_0]$.  The
stage set $\omega_3$ is partitioned into three cofinal bookkeeping classes:
\[
I_{\mathrm{MA}},\qquad I_{\mathrm{WO}},\qquad I_\Sigma.
\]
Stages in $I_{\mathrm{MA}}$ force Martin's Axiom, stages in $I_{\mathrm{WO}}$
repeat the FFDZ wellorder coding, and stages in $I_\Sigma$ carry out the global
$\Sigma$-uniformization construction.  The classes
$\mathcal L_{\mathrm{MA}},\mathcal L_{\mathrm{WO}},\mathcal L_\Sigma$ from
Section~\ref{sec:psi-three} are classes of preliminary apparatus blocks, while
$I_{\mathrm{MA}},I_{\mathrm{WO}},I_\Sigma$ are classes of later c.c.c. iteration
stages.  A stage in $I_{\mathrm{WO}}$ chooses a fresh block from
$\mathcal L_{\mathrm{WO}}$, and a stage in $I_\Sigma$ chooses a fresh block from
$\mathcal L_\Sigma$. The two partitions serve different purposes and should not be conflated.

The classes
\[
\mathcal L_{\mathrm{MA}},\quad
\mathcal L_{\mathrm{WO}},\quad
\mathcal L_{\Sigma}
\]
classify the bookkeeping requirements handled by the construction. That is, they
determine whether a given stage of the bookkeeping concerns side forcing for
Martin's Axiom/cardinal characteristics, coding for the projective wellorder,
or global \(\Sigma\)-uniformization.

By contrast, the sets
\[
I_{\mathrm{MA}},\quad
I_{\mathrm{WO}},\quad
I_{\Sigma}
\]
form a partition of the actual iteration coordinates. They determine where in the
iteration the corresponding forcing operations are allowed to act.

This distinction is important because the construction repeatedly needs fresh
unused FFDZ blocks. The bookkeeping may decide that a certain request is of
uniformization type, i.e. belongs to \(\mathcal L_{\Sigma}\), but the associated
coding must still be carried out on a previously unused coordinate from
\(I_{\Sigma}\). Similarly, wellorder coding is always confined to
\(I_{\mathrm{WO}}\), ensuring that the two coding mechanisms cannot interfere.

Thus the \(\mathcal L\)-classes organize the semantic bookkeeping, whereas the
\(I\)-classes organize the geometric support structure of the iteration.

We also fix, once and for all, a canonical $L$-definable ordering of all nice
names for reals appearing in initial segments of the later finite-support
iteration.  Names from shorter initial segments precede names which genuinely
require longer initial segments, and ties are broken by the canonical wellorder
of $L$.  For a real $x$ present at a stage and an ordinal $\xi<\omega_3$, the
notation
\[
(x,y_\xi,a^\xi_0)
\]
refers to the interpretation of the $\xi$-th pair of canonical names for reals,
with the first coordinate fixed to be $x$.  We arrange that the first pair is
$(0,0)$, so $(x,y_0,a^0_0)=(x,0,0)$.  In the final model this name ordering is
exactly the ordering which underlies the stage-of-appearance wellorder.  Thus no
uniformization stage needs access to the final relation $<_G$; it only uses the
canonical names available at that stage.

The construction maintains an external increasing sequence
\[
\langle\mathcal F_\alpha:\alpha\leq\omega_3\rangle
\]
of forbidden canonical names for real codes of uniformization tuples.  If a
canonical name $\dot z$ belongs to $\mathcal F_\alpha$, then no later
$\Sigma$-stage is allowed deliberately to code any real forced equal to
$\dot z$.  When a stage declares a tuple forbidden, the construction adds the
least canonical name for the recursive real code of that tuple to all later
$\mathcal F_\beta$.  This is the only reservation mechanism used below; there is
no separate external forbidden set attached to a stage.

\begin{definition}[Support below a preliminary coordinate]\label{def:support-below}
Let $\alpha<\omega_3$.  A $\Pp_\alpha$-name $\dot{\mathbb Q}$, or a condition
appearing in such a name, is \emph{supported below $\alpha$ in the preliminary
coordinate} if its preliminary mixed-support component uses only apparatus
coordinates $<\alpha$.  Equivalently, the name is read in the complete subforcing
generated by the preliminary coordinates below $\alpha$ together with the
already constructed c.c.c. iteration below $\alpha$, and it is trivial on all
fresh apparatus blocks $\geq\alpha$.
\end{definition}

This definition is used to keep future branch-versus-specialize blocks fresh.
In particular, an MA stage may use arbitrary c.c.c. forcing only when the name is
read below the current preliminary coordinate in this sense.

\subsection{MA stages}

If $\alpha\in I_{\mathrm{MA}}$ and the bookkeeping presents a $\Pp_\alpha$-name
$\dot{\mathbb Q}$ for a c.c.c. forcing, together with the relevant dense sets,
and $\dot{\mathbb Q}$ is supported below $\alpha$ in the sense of
Definition~\ref{def:support-below}, then let $\dotQ_\alpha=\dot{\mathbb Q}$.
Otherwise $\dotQ_\alpha$ is trivial.  The bookkeeping is chosen so that every
c.c.c. forcing of size $<\aleph_3$ with $<\aleph_3$ many dense sets appears
cofinally often after its name has bounded preliminary support.

The preservation of future coding blocks at these stages uses
Lemma~\ref{lem:jech-subtlety}: c.c.c. forcing on the side does not destroy the
Suslinity of a fresh Jech tree in a product with that Jech coordinate.  Thus MA
stages do not accidentally create branch/specialize codes on unused blocks.

\subsection{Wellorder stages}

If $\alpha\in I_{\mathrm{WO}}$, the bookkeeping presents a pair of names for reals
in $\Wground[G_\alpha]$.  If the names are supported below the fresh wellorder
block assigned to $\alpha$, the iteration performs the FFDZ wellorder coding on
that block: it specializes the trees whose indices lie in
$\Delta_{\mathrm{WO}}(x*y)$, adds branches through the remaining trees, and then
almost-disjointly codes the branch/specialization data by a real.  If the
bookkeeping entry is not of this form, the stage is a dummy wellorder stage.

By Lemma~\ref{lem:wo-sigma13}, the union of these stages yields a lightface
$\Delta^1_3$ wellorder $<_G$ of the reals in the final model.

\subsection{Uniformization stages}\label{subsec:uniformization-stages}

The uniformization stages follow the architecture of
\cite{BPFA_and_global_Sigma-uniformization}.  They use the base predicate $\PsiThree$ to code,
or to forbid coding, single real codes which externally compress possibly
infinite conjunctions or disjunctions of $\Pi^1_2$ or $\Sigma^1_2$ matrix tests.

Let the bookkeeping at a stage $\alpha\in I_\Sigma$ present
\[
(e,\dot x,\xi,\dot b_1,\dot b_2,\ldots),
\]
where $e$ codes a projective formula, $\dot x$ and the $\dot b_i$ are
$\Pp_\alpha$-names for reals, and $0<\xi<\omega_3$ is a candidate index in the
canonical name ordering just described.  Let $x=(\dot x)^{G_\alpha}$ and
$b_i=(\dot b_i)^{G_\alpha}$.  Let
\[
(x,y_\eta,a^\eta_0)_{\eta<\omega_3}
\]
be the evaluated candidate list obtained from the canonical names at the stage;
this notation is purely name-based and does not presuppose that the final
wellorder relation $<_G$ is available in $\Wground[G_\alpha]$.

Suppose first that $e$ codes an odd-level formula
\[
\theta_e(x,y)\equiv \exists a_0\forall a_1\exists a_2\cdots\exists a_{2m-2}
\psi_e(x,y,a_0,a_1,\ldots,a_{2m-2}),
\]
where $\psi_e$ is $\Pi^1_2$.  Decode
\[
\pi_\xi(b_i)=\langle b_i^\eta:\eta<\xi\rangle.
\]
Let $z^O_\xi$ be the recursive real code of
\[
(\#\psi_e,x,y_\xi,a^\xi_0,b_1,\ldots,b_{2m-2}),
\]
and let $\dot z^O_\xi$ be its least canonical name at the stage.  If
\[
\Wground[G_\alpha]\models
\forall\eta<\xi\;\neg\psi_e(x,y_\eta,a^\eta_0,b_1^\eta,\ldots,b_{2m-2}^\eta),
\tag{O1}
\]
and $\dot z^O_\xi\notin\mathcal F_\alpha$, then $\dotQ_\alpha$ is the block
coding forcing which deliberately codes $z^O_\xi$, equivalently makes
\[
\PsiThree(\#\psi_e,x,y_\xi,a^\xi_0,b_1,\ldots,b_{2m-2})
\]
true.  If (O1) fails, the stage is trivial and $\dot z^O_\xi$ is added to the
forbidden set.  If (O1) holds but $\dot z^O_\xi\in\mathcal F_\alpha$, the stage
is also trivial.  Thus once an earlier candidate succeeds, all later compressed
tuples witnessing that fact are forbidden rather than coded.

For even-level formulas the rule is dual.  Suppose
\[
\theta_e(x,y)\equiv \exists a_0\forall a_1\exists a_2\cdots\forall a_{2m-3}
\psi_e(x,y,a_0,a_1,\ldots,a_{2m-3}),
\]
where $\psi_e$ is $\Sigma^1_2$.  Let $z^E_\xi$ be the recursive real code of
\[
(\#\psi_e,x,y_\xi,a^\xi_0,b_1,\ldots,b_{2m-3}).
\]
If
\[
\Wground[G_\alpha]\models
\forall\eta<\xi\;\neg\psi_e(x,y_\eta,a^\eta_0,b_1^\eta,\ldots,b_{2m-3}^\eta),
\tag{E1}
\]
then the stage is trivial and the least canonical name for $z^E_\xi$ is added to
$\mathcal F_{\alpha+1}$.  If (E1) fails and the code is not already forbidden,
then the stage deliberately codes $z^E_\xi$ into a fresh uniformization block.
If the code is already forbidden, the stage is trivial.

\begin{lemma}\label{lem:iteration-ccc}
The iteration just defined is a finite-support c.c.c. iteration of length
$\omega_3$.  Its final model satisfies
\[
\MA+2^{\aleph_0}=\aleph_3.
\]
\end{lemma}

\begin{proof}
MA stages are c.c.c. by definition.  Wellorder and uniformization block stages
are c.c.c. by Baumgartner's theorem for specializing Suslin trees together with
the standard fact that forcing with a Suslin tree is c.c.c.; the subsequent
almost-disjoint real coding is c.c.c.  Lemma~\ref{lem:no-accidental-codes}
ensures that the trees used at a fresh block are still Suslin.  Finite-support
iterations of c.c.c. forcings are c.c.c.

The bookkeeping on $I_{\mathrm{MA}}$ meets every c.c.c. forcing of size
$<\aleph_3$ and every family of $<\aleph_3$ many dense sets once its name is
supported below a sufficiently large preliminary coordinate.  Hence $\MA$ holds
in the final model.  The iteration has length $\omega_3$, size $\aleph_3$, and
adds reals cofinally often, so $2^{\aleph_0}=\aleph_3$.
\end{proof}

\section{Derived projective predicates and copied infinitary tests}\label{sec:uniformization}

We now verify that the uniformization stages produce the promised same-level
predicates.  We use Shoenfield absoluteness in the following standard form:
whenever the relevant real parameters belong to both transitive models under
comparison, the $\Pi^1_2$ and $\Sigma^1_2$ matrix statements considered below
have the same truth value in the intermediate extension in which the
bookkeeping decision is made and in the final extension.  No separate lemma is
needed for this fact.

\begin{lemma}\label{lem:bookkeeping-complete}
Let $e$, $x$, a candidate index $\xi<\omega_3$, and finitely many reals
$b_1,\ldots,b_k$ belong to the final extension.  Suppose the canonical candidate
names for $(x,y_\eta,a^\eta_0)_{\eta\leq\xi}$ are fixed as above.  Then there is
some stage $\alpha\in I_\Sigma$ after all these parameters and the relevant
compressed sequence names are supported such that the bookkeeping at $\alpha$
presents exactly the corresponding request
\[
(e,\dot x,\xi,\dot b_1,\ldots,\dot b_k).
\]
At that stage the truth value of every displayed $\Pi^1_2$ or $\Sigma^1_2$ matrix
statement with these real parameters is the same as in the final extension.
Consequently the construction either codes the corresponding tuple or places its
canonical code name into the forbidden set according to the rules of
Section~\ref{subsec:uniformization-stages}.
\end{lemma}

\begin{proof}
By Lemma~\ref{lem:bounded-sequence-names}, the finitely many real parameters and
the compressed sequence names used by the maps $\pi_\xi$ are all read by some
proper initial segment of the iteration.  Since $I_\Sigma$ is cofinal and the
bookkeeping on $I_\Sigma$ repeats every possible finite request cofinally often,
there is a later $\Sigma$-stage presenting precisely these names.  The matrix
statements which the stage has to evaluate are $\Pi^1_2$ or $\Sigma^1_2$ with the
listed real parameters.  Shoenfield absoluteness gives equality of their truth
values between the intermediate transitive extension at that stage and the final
extension.  The last sentence is then just the definition of the uniformization
stage, together with the monotonicity of the forbidden sets.
\end{proof}

For a tuple $t$ of reals, write
\[
\operatorname{Coded}(t)
\]
for $\PsiThree(z_t)$, where $z_t$ is the fixed recursive real coding $t$.  Write
$\operatorname{NotCoded}(t)$ for its negation.  Since $\PsiThree$ is
$\Sigma^1_3$, these predicates supply the base alternation used in the following
normal forms.

\begin{lemma}[Later candidates are excluded]\label{lem:later-excluded}
Work in the final model and fix an odd-level formula
\[
\theta_e(x,y)\equiv \exists a_0\forall a_1\exists a_2\cdots\exists a_{2m-2}
\psi_e(x,y,a_0,
\ldots,a_{2m-2}),
\]
with $\psi_e\in\Pi^1_2$.  Suppose $\xi$ is such that
\[
\forall a_1\exists a_2\cdots\exists a_{2m-2}
\psi_e(x,y_\xi,a^\xi_0,a_1,\ldots,a_{2m-2})
\]
holds.  Then for every $\beta>\xi$,
\[
\forall a_1\exists a_2\cdots\exists a_{2m-2}\;
\operatorname{NotCoded}(\#\psi_e,x,y_\beta,a^\beta_0,a_1,\ldots,a_{2m-2}).
\tag{1}
\]
In the even-level case, with $\psi_e\in\Sigma^1_2$, the corresponding conclusion
is obtained by replacing $\operatorname{NotCoded}$ in (1) by
$\operatorname{Coded}$.
\end{lemma}

\begin{proof}
We prove the odd case.  Fix $\beta>\xi$ and let $a_1$ be arbitrary.  Decode
$a_1$ externally as
\[
\pi_\beta(a_1)=\langle a_1^\eta:\eta<\beta\rangle.
\]
Since the candidate $\xi$ is successful, there are witnesses
$a_2^\xi,a_3^\xi,\ldots,a_{2m-2}^\xi$ satisfying the remaining alternation for
\[
\psi_e(x,y_\xi,a^\xi_0,a_1^\xi,a_2^\xi,\ldots,a_{2m-2}^\xi).
\]
Choose arbitrary values for the other coordinates and compress the resulting
families by the maps $\pi_\beta$ to obtain reals $a_2,\ldots,a_{2m-2}$.  By
Lemma~\ref{lem:bookkeeping-complete}, the request for the tuple
\[
(\#\psi_e,x,y_\beta,a^\beta_0,a_1,\ldots,a_{2m-2})
\]
is presented at some sufficiently late $\Sigma$-stage.  At that stage the test
(O1) fails, because the coordinate $\eta=\xi$ makes the $\Pi^1_2$ matrix true
there exactly as in the final model.  Hence the construction places the
canonical code name for this tuple into the forbidden set.  By monotonicity of
forbiddenness it is never later deliberately coded, and by
Lemma~\ref{lem:block-correctness} it is not $\PsiThree$-coded in the final model.
This gives (1).  The even case is the same argument with the dual coding rule.
\end{proof}

\begin{lemma}[The least successful candidate is selected]\label{lem:least-selected}
Work in the final model.  In the odd-level case, suppose $\xi>0$ is least such
that
\[
\forall a_1\exists a_2\cdots\exists a_{2m-2}
\psi_e(x,y_\xi,a^\xi_0,a_1,\ldots,a_{2m-2})
\]
holds.  Then
\[
\exists a_1\forall a_2\cdots\forall a_{2m-2}\;
\operatorname{Coded}(\#\psi_e,x,y_\xi,a^\xi_0,a_1,\ldots,a_{2m-2}).
\tag{2}
\]
In the even-level case the corresponding conclusion is obtained by replacing
$\operatorname{Coded}$ in (2) by $\operatorname{NotCoded}$.
\end{lemma}

\begin{proof}
Again we prove the odd case.  Since $\xi$ is least successful, for each
$\eta<\xi$ there are counter-witnesses arranged according to the dual quantifier
pattern witnessing the failure of
\[
\forall a_1\exists a_2\cdots\exists a_{2m-2}
\psi_e(x,y_\eta,a^\eta_0,a_1,\ldots,a_{2m-2}).
\]
Using Convention~\ref{conv:compression-maps}, compress the first layer of these
counter-witnesses into a single real $a_1$.  Now let $a_2$ be arbitrary.  Decode
it into the corresponding family, choose the next layer of counter-witnesses for
each $\eta<\xi$, compress that layer into $a_3$, and continue through the
finite alternating pattern.  The result is that, for every play of the universal
variables on the coded tuple, all earlier coordinates $\eta<\xi$ make the matrix
false.

Thus, for every choice of the universal variables in the displayed tuple,
Lemma~\ref{lem:bookkeeping-complete} gives a sufficiently late $\Sigma$-stage at
which the corresponding request is presented and the test (O1) holds.  If the
code had already been forbidden, then it would have been forbidden by an earlier
application of the same rule, which would require some earlier coordinate to
make the matrix true; this contradicts the compressed family of earlier
failures.  Hence the construction codes the tuple into a fresh uniformization
block.  Lemma~\ref{lem:block-correctness} then says that it is $\PsiThree$-coded
in the final model.  This proves (2).  The even-level case is the dual rule: the
same compressed family of earlier failures causes the tuple to be forbidden
rather than coded, giving the displayed statement with $\operatorname{NotCoded}$.
\end{proof}

\section{Uniformization}\label{sec:uniformization-final}

We now define the final uniformizing predicates explicitly.  No separate
acceptance predicate is needed.

For an odd-level formula
\[
\theta_e(x,y)\equiv \exists a_0\forall a_1\exists a_2\cdots\exists a_{2m-2}
\psi_e(x,y,a_0,a_1,\ldots,a_{2m-2}),
\quad \psi_e\in\Pi^1_2,
\]
let $U_e(x,y)$ be the disjunction of the following two clauses:
\begin{enumerate}[label=(\roman*)]
\item $y=0$ and
\[
\forall a_1\exists a_2\cdots\exists a_{2m-2}
\psi_e(x,0,0,a_1,\ldots,a_{2m-2});
\]
\item the clause in (i) fails and there is $a_0$ such that
\[
\forall a_1\exists a_2\cdots\exists a_{2m-2}
\psi_e(x,y,a_0,a_1,\ldots,a_{2m-2})
\]
and
\[
\neg\Bigl[\forall a_1\exists a_2\cdots\exists a_{2m-2}\;
\operatorname{NotCoded}(\#\psi_e,x,y,a_0,a_1,\ldots,a_{2m-2})\Bigr].
\]
\end{enumerate}
This is a $\Sigma^1_{2m+1}$ predicate.  The last negated bracket has exactly the
same projective complexity as the original existential-leading formula because
$\operatorname{NotCoded}$ is the negation of the base $\Sigma^1_3$ predicate and
appears under the corresponding alternating quantifier pattern.

For an even-level formula
\[
\theta_e(x,y)\equiv \exists a_0\forall a_1\exists a_2\cdots\forall a_{2m-3}
\psi_e(x,y,a_0,a_1,\ldots,a_{2m-3}),
\quad \psi_e\in\Sigma^1_2,
\]
let $U_e(x,y)$ be the disjunction of the following two clauses:
\begin{enumerate}[label=(\roman*)]
\item $y=0$ and
\[
\forall a_1\exists a_2\cdots\forall a_{2m-3}
\psi_e(x,0,0,a_1,\ldots,a_{2m-3});
\]
\item the clause in (i) fails and there is $a_0$ such that
\[
\forall a_1\exists a_2\cdots\forall a_{2m-3}
\psi_e(x,y,a_0,a_1,\ldots,a_{2m-3})
\]
and
\[
\neg\Bigl[\forall a_1\exists a_2\cdots\forall a_{2m-3}\;
\operatorname{Coded}(\#\psi_e,x,y,a_0,a_1,\ldots,a_{2m-3})\Bigr].
\]
\end{enumerate}
This is a $\Sigma^1_{2m}$ predicate by the same complexity calculation, with the
parity of the base coded/non-coded clause reversed.

\begin{theorem}[Main theorem]\label{thm:main}
There is a generic extension of $L$ satisfying
\[
\MA+2^{\aleph_0}=\aleph_3
\]
such that the reals have a lightface $\Delta^1_3$ wellorder and, for every
$n\geq 2$, every $\Sigma^1_n$ set of pairs of reals admits a $\Sigma^1_n$
uniformization.
\end{theorem}

\begin{proof}
Let $G$ be generic for the iteration described above.  Lemma~\ref{lem:iteration-ccc}
gives $\MA$ and $2^{\aleph_0}=\aleph_3$.  The stages in $I_{\mathrm{WO}}$ reproduce
the FFDZ proof, and Lemma~\ref{lem:wo-sigma13} gives a lightface $\Delta^1_3$
wellorder $<_G$ of the reals.

The case $n=2$ follows in $\ZFC$ from the Kondo--Novikov theorem.  Let $n\geq3$
and let $\theta_e(x,y)$ be a $\Sigma^1_n$ formula.  We discuss the odd case; the
even case is dual.  If the $x$-section is empty, no value is assigned.  If it is
nonempty and the initial triple $(x,0,0)$ is successful, then clause (i) defines
$U_e(x,0)$.  Otherwise let $\xi>0$ be least such that some leading witness
$a^\xi_0$ makes the tail matrix true for the candidate $y_\xi$.  Lemma~\ref{lem:least-selected}
shows that the second clause of the definition holds for $(x,y_\xi)$.
For every $\beta>\xi$, Lemma~\ref{lem:later-excluded} gives the full
non-coding tail, so the negated bracket in the definition fails.  For
$\beta<\xi$, the original matrix tail fails by minimality of $\xi$.  Hence the
relation $U_e$ is single-valued, has domain equal to the projection of
$\theta_e$, and is contained in $\theta_e$.  Its projective complexity is
$\Sigma^1_n$ by the displayed normal form.  Thus $U_e$ is the desired
same-level uniformization.
\end{proof}

\section{A cardinal-characteristic variant}\label{sec:cardinal-variant}

The method is modular.  If full Martin's Axiom is replaced by the
cardinal-characteristic bookkeeping of Fischer--Friedman--Zdomskyy, the same
wellorder and uniformization stages can be retained.

\begin{theorem}\label{thm:cardinal-variant}
There is a generic extension in which $2^{\aleph_0}=\aleph_3$, the
$\Sigma^1_n$-uniformization property holds for every $n\geq2$, the reals have a
lightface $\Delta^1_3$ wellorder, and
\[
\mathfrak p=\mathfrak b=\aleph_2<\mathfrak a=\mathfrak s=\mathfrak c=\aleph_3.
\]
\end{theorem}

\begin{proof}
Use the same preliminary FFDZ apparatus and the same wellorder and global
$\Sigma$-uniformization bookkeeping.  Replace the class $I_{\mathrm{MA}}$ by the
successor-stage bookkeeping used in \cite[Section 3]{FISCHER2013763}.  Below
$\omega_2$ add a $<^*$-increasing unbounded scale
$H=\langle h_\xi:\xi<\omega_2\rangle$ by Hechler stages.  Above $\omega_2$, use
cofinal classes of successor stages for the Fischer--Steprans forcings adding
unsplit reals while preserving $H$, for Brendle's forcings destroying small mad
families while preserving $H$, and for all $\sigma$-centered posets of size at
most $\aleph_1$.

The additional wellorder and uniformization stages are of the same small c.c.c.
branch-versus-specialize type as in the main construction.  They have size at
most $\aleph_1$ at the relevant stages, or belong to the closed preliminary
apparatus, and hence preserve the unboundedness of $H$ by the same preservation
argument used in \cite{FISCHER2013763}.  Lemma~\ref{lem:jech-subtlety} again
prevents the additional c.c.c. stages from accidentally creating codes on fresh
Jech blocks.  Therefore the exactness of $\PsiThree$ and the proof of global
$\Sigma$-uniformization are unchanged.

The cardinal-characteristic computation is the one from
\cite[Corollary 3.1]{FISCHER2013763}.  The iteration has continuum $\aleph_3$;
forcing with all $\sigma$-centered posets of size at most $\aleph_1$ gives
$\MA_{<\omega_2}(\sigma\text{-centered})$, hence $\mathfrak p=\aleph_2$ by Bell's
theorem \cite{Bell}.  Since $\mathfrak p\leq\mathfrak b$ and $H$ remains
unbounded of size $\aleph_2$, we get $\mathfrak b=\aleph_2$.  Cofinal
Fischer--Steprans stages give $\mathfrak s=\aleph_3$, and cofinal Brendle stages
destroy every name for a maximal almost disjoint family of size at most
$\aleph_2$, giving $\mathfrak a=\aleph_3$.
\end{proof}

\section{Questions}\label{sec:questions}

We close with several questions suggested by the construction.

The first question concerns the compatibility of Martin's Axiom with
\(\Pi\)-side uniformization.  The present paper shows that Martin's Axiom is
compatible with global \(\Sigma\)-uniformization.  It is natural to ask whether
one can obtain a dual phenomenon.

\begin{question}
Is it consistent that Martin's Axiom holds and the boldface
\(\Pi^1_3\)-uniformization property holds?
\end{question}

This question seems innocent, but it is not addressed by the available
technology.  In fact, stronger forcing axioms can push in the opposite direction:
by \cite{HoffelnerForcingAxioms}, \(\BPFA\) together with the anti-large-cardinal
assumption \(\omega_1=\omega_1^L\) implies boldface
\(\Sigma^1_3\)-uniformization.  Since same-level uniformization for a pointclass
is incompatible with same-level uniformization for the dual pointclass, this
shows that the analogous question for \(\BPFA+\omega_1=\omega_1^L\) has a
negative answer.  The case of \(\MA\), however, remains open.

A related problem is whether the theorem above can be obtained without producing
a projective wellorder.

\begin{question}
Is there a model of global \(\Sigma\)-uniformization in which there is no
projective wellorder of the reals at the corresponding low level?
\end{question}

All known forcing constructions of global \(\Sigma\)-uniformization, including
the present one and the construction in
\cite{BPFA_and_global_Sigma-uniformization}, are compatible with, and in practice
produce, a projective wellorder.  Removing the wellorder from the construction
would require a substantially different method for making the least choices
projectively visible.

One may also ask whether the \(\Sigma\)-side pattern can be combined with richer
cardinal-characteristic configurations.

\begin{question}
Which assignments of values to the standard cardinal characteristics of the
continuum are compatible with global \(\Sigma\)-uniformization and a projective
wellorder of the reals?
\end{question}

Finally, the present work belongs to a broader attempt to force prescribed
patterns of projective uniformization.

\begin{question}
Let \(E\subseteq\omega\).  Is it consistent that
\(\Sigma^1_n\)-uniformization holds exactly for those \(n\in E\), subject to the
trivial restrictions coming from ZFC implications such as the Kondo--Novikov
theorem at level \(2\)?
\end{question}

Partial results in this direction are obtained in
\cite{HoffelnerUpperSigma,HoffelnerSigma34}, but a general pattern theorem seems
to require new ideas.

\bibliographystyle{plain}
\bibliography{references}

\end{document}